\documentclass[12pt]{amsart}
\usepackage{amsmath}%
\usepackage{amsfonts}%
\usepackage{fullpage}
\usepackage[DIV=calc]{typearea}
\usepackage{eco}
\usepackage[hug,balance,dpi=1200,Postscript=dvips]{diagrams}
\usepackage{mparhack}
\usepackage[rm,it,tiny,compact,center]{titlesec}
\usepackage{graphicx}
\usepackage{rotating}
\usepackage{perpage}
\usepackage[x11names]{xcolor}
\usepackage{xspace}

\newtheorem{theorem}{Theorem}[section]
\theoremstyle{plain}
\titlelabel{ \rm \newstylenums\thetitle. \/}

\newtheorem{corollary}[theorem]{Corollary}

\newtheorem{lemma}[theorem]{Lemma}
\newtheorem{proposition}[theorem]{Proposition}
\newtheorem{remark}[theorem]{Remark}

\MakePerPage[2]{footnote}
\def \endit{\quad\vrule height 1.9ex width 1.174ex depth .2ex \par} 
\renewenvironment {proof}{{\par\bf Proof:{\enspace}}}{\endit\smallskip}
\newarrow{Corresponds}<---> 
\everymath{\displaystyle}
\begin{document}
\title[Hopf C*-algebras]{On convolution products and automorphisms in  Hopf C*-algebras}
\author{Dan Z. Ku\v{c}erovsk\'{y}} 
\address
{University of New Brunswick \newline%
\indent Canada\qquad E3B 5A3}%
\email{dkucerov@unb.ca}%
\thanks{We thank NSERC for financial support. \\
{\textit{mailing address:} \\
Dept. of Math.,\newline
 University of New Brunswick \newline%
Canada\qquad E3B 5A3}}
\subjclass{Primary 47L80, 16T05; Secondary 47L50, 16T20 } %
\keywords{Hopf algebras, C*-algebras}\dedicatory{}

\newcommand{\Mult}{{\mathcal M}} 
\newcommand{\Hgns}{\ensuremath{{\mathcal H}_{\text{\tiny GNS}}}} 
\renewcommand{\H}{\ensuremath{{\mathcal H}}} 
\newcommand{\C}{\mathbb C}
\newcommand{\N}{\mathbb N}
\newcommand{\Z}{\mathbb Z}
\newcommand{\R}{\mathbb R}
\newcommand{\quaternions}{\mathbb H}
\newcommand{\M}[1]{M_{#1}(\C)} 
\newcommand{\Tr}{\mbox{\rm Tr}}
\newcommand{\BH}{\ensuremath{B(\H)}} 
\newcommand{\tensor}{\otimes} 
\newcommand{\BHH}{\ensuremath{B(\H\tensor\H)}} 
\newcommand{\entrelacement}{\/\textit{entrelacement\xspace}} 
\newcommand{\isom}{\cong}
\newcommand{\Sh}{\ensuremath{\widehat{S}} } 
\newcommand{\Ah}{\ensuremath{\widehat{A}} } 
\newcommand{\Bh}{\ensuremath{\widehat{B}} } 
\newcommand{\AhA}{\ensuremath{\Ah\tensor A}}
\newcommand{\BhB}{\ensuremath{\Bh\tensor B}}
\newcommand{\uh}{\ensuremath{\widehat{u}}} \newcommand{\Uh}{\ensuremath{\widehat{U}}} 
\newcommand{\fancyK}{\ensuremath{\widetilde{K}}} 
\renewcommand{\fancyK}{K}
\newcommand{\antipode}{\ensuremath{\kappa}} 
\newcommand{\antipodeh}{\ensuremath{\widehat\kappa}} 
\newcommand{\ksymmetric}{$\antipode$-symmetric\xspace}
\newcommand{\kpositive}{$\antipode$-positive\xspace}
\newcommand{\kpositivity}{$\antipode$-positivity\xspace}
\newcommand{\counit}{\varepsilon} 
\newcommand{\counith}{\widehat{\varepsilon}}
\newcommand{\F}{\mbox{\ensuremath{\mathcal F}}} 
\newcommand{\Fh}{\mbox{\ensuremath{\widehat{\mathcal F}}}} 
\newcommand{\inv}{\ensuremath{{}^{-1}}}
\newcommand{\Id}{\mbox{\rm Id}}
\newcommand{\Ad}{\mbox{\rm Ad}}  
\newcommand{\Fix}{\mbox{\rm Fix}}  
\newcommand{\CoZ}{\mbox{\rm CoZ}} 
\newcommand{\CoZr}{\CoZ_{\R}} 
\newcommand{\BiU}{\mbox{\rm Bi\hspace{1pt}\ensuremath{\mathcal{U}}}} 
\newcommand{\cofixe}{\/\textit{cofix\'e }} 
\newcommand{\fixe}{\/\textit{fix\'{e} }}
\newcommand{\Inn}{\mbox{\rm Inn$_H$}}  
\newcommand{\Ug}{{\mathcal U}}
\newcommand{\arrow}{\longrightarrow}
\renewcommand{\dim}{{\mbox{\rm Dim}_\C}}
\newcommand{\ps}{pseudosimple\xspace}
\newcommand{\pscity}{pseudosimplicity\xspace}
\newcommand{\commutes}{\copyright} 
\newcommand{\product}[1][\empty]{\ensuremath{\pmb{\boldsymbol{\cdot\!}}_#1}}
\newcommand{\coproduct}[1][\empty]{\ensuremath{\delta_#1}}
\newcommand{\norm}[1]{\left\|#1\right\|}
\newcommand{\lfnorm}[1]{\norm{#1}_{\scriptscriptstyle{lf}}}
\newcommand{\tracenorm}[1]{\norm{#1}_{\scriptscriptstyle{1}}}
\newcommand{\Ksymbol}{\ensuremath{\fancyK\mbox{-symbol}}\xspace}
\newcommand{\Ksymbols}{{\Ksymbol}\mbox{s}\xspace}
\newcommand{\convolution}{\diamond} \newcommand{\convolute}{\convolution}  \newcommand{\convolve}{\convolution} 
\newcommand{\ip}[2]{\ensuremath{\left\langle #1,#2\right\rangle}}
\newcommand{\boxproduct}{\,\square\,} 
\newcommand\comment[1]{\-\marginpar[\raggedleft\footnotesize\it\textcolor{Sienna4}{#1}]
{\raggedright\footnotesize\it\textcolor{Sienna4}{#1}}}
\renewcommand\comment[1]{} 
\newcommand{\compose}{\circ}
 \newcommand{\compact}{\mathcal K}
\newcommand{\Cu}{\ensuremath{{\mathcal C}\hspace{-.75pt}u}}
\definecolor{refkey}{cmyk}{0.1,0.1,1,0}
\definecolor{labelkey}{cmyk}{0,0.2,1,0.1}
\begin{abstract}
 We obtain two characterizations of the bi-inner Hopf *-automorphisms of a finite-dimensional Hopf C*-algebra, by means of an analysis of the structure of convolution products in this class of Hopf C*-algebra.  
\end{abstract}
\maketitle

\section{Introduction} 
A Hopf algebra is a bi-algebra with an antipode map. 
It seems interesting to classify Hopf C*-algebras by $K$-theoretical methods similar to those used in the purely C*-algebraic setting (the Elliott program). The idea here  is to find a functor from some class of Hopf C*-algebras to a classifying category. In general, such a functor will annihilate some precise class of Hopf algebra automorphisms, and so it becomes interesting to understand the detailed structure of Hopf *-automorphisms in various cases. In this paper, we address the case of  bi-inner Hopf *-automorphisms. These are the Hopf *-automorphisms of a Hopf C*-algebra that are inner as algebra automorphisms, both in the dual algebra and in the given algebra. We obtain two characterizations of the group of bi-inner  Hopf *-automorphisms  for the case of finite-dimensional Hopf C*-algebras (Theorems \ref{th:bi.inner} and \ref{th:biU}). We also develop techniques of independent interest, involving positivity-preserving maps, faithful linear functionals, and other related concepts.

\section{Faithful positive linear functionals and bilateral maps}

Our basic setting is a C*-algebra that is a compact and discrete Hopf algebra, thus being finite-dimensional as a C*-algebra.  
Our notation is based on that of \cite{BS}, denoting co-products by $\coproduct,$ antipode by $\antipode,$ and  co-unit by $\counit.$ The real algebra of \ksymmetric elements is the algebra of elements that are fixed under the antipode composed with the C*-algebraic involution.  The co-centre denotes the set of all co-commutative elements. 

 We recall that a faithful linear functional $f$ is a linear functional such that $f(ab)=0$ for all $b$ implies that $a$ is zero, and such that $f(ab)=0$ for all $a$ implies that $b$ is zero. Faithful linear functionals are well-understood in the positive case, but some of their familar properties do not generalize well beyond the positive case. For example, a faithful positive linear function $f$ composed with a positive injective linear map $\psi$ is again faithful. However, if we consider the composition $f\compose\psi$ of a general faithful linear functional with, say, an injective *-homomorphism $\psi,$ the composition need not be faithful. The problem is that although, if the element $a$ is not zero, there does exist an element $b$ such that $f(\psi(a)b)$ is nonzero, it does not follow that $b$ can be taken to be in the range of $\psi.$ It thus seems natural to consider the self-adjoint case, and we show below that the composition of a self-adjoint faithful linear functional with an injective *-homomorphism is again a self-adjoint faithful linear functional. 

In the finite-dimensional case, faithful linear functionals are exactly those that can be written as $\tau(v\cdot)$ where $\tau$ is the tracial Haar state and $v$ is an invertible element of the algebra. 
The Jordan decomposition of self-adjoint linear functionals gives us a lemma that we prove for the reader's convenience.
\begin{lemma}A faithful self-adjoint linear functional on a finite-dimensional C*-algebra $A$ decomposes as $f_1-f_2$ where the $f_i$ are positive and orthogonal. There is a projection $P$ such that $f_1$ is faithful on $PA$ and $f_2$ is faithful on $(1-P)A$.\end{lemma}
\begin{proof} We are given $f=\tau(v\:\cdot)$ where $v$ is invertible and self-adjoint. Polar decomposition gives a symmetry $U$ with $v=|v|U.$ But then let $P:=\frac{1+U}{2},$ and let $f_1=f(P\:\cdot)$, $f_2=-f((1-P)\:\cdot).$
\end{proof}
\begin{corollary}Let $\coproduct{}\colon A\arrow B$ be an injective $*$-homomorphism, and let $f$ be a faithful self-adjoint linear functional on $B$. Then $f\compose\coproduct{}$ is a faithful linear functional on $A.$\label{cor:pullback.of.flf}
\end{corollary}
\begin{proof} The above lemma provides a decomposition into complemented right ideals, $B=B_1\oplus B_2$ and linear functionals $f_i$ such that $f_i$ is faithful and positive on $B_i.$ But then $A_i:=\coproduct{}\inv(B_i)$ is a right ideal of $A$, and $f_i \compose\coproduct{} $ is a positive faithful linear functional on $A_i$. Thus $f\compose\coproduct{}=f_1\compose\coproduct{}-f_2\compose\coproduct{}$ is a self-adjoint faithful linear functional on $A.$
\end{proof}

We recall that on a Hopf algebra there exists a  Fourier transform, defined  by $\beta(a,\F(b))=\tau(ab),$ where $\beta$ is the pairing, and $\tau$ is the Haar state. This lets us define an operator-valued convolution product by $a\convolve b := \F\inv(\F(a)\F(b)).$ 
\begin{lemma}Let $c$ and $y$ be self-adjoint and invertible elements of a finite-dimensional Hopf C*-algebra, $A.$ Then $c\convolve y$ is self-adjoint and invertible.\label{lem:sa.inverse.case}\end{lemma}
\begin{proof} An operator $x$ is invertible if and only if the Fourier transform $\F\colon A\arrow\Ah$ maps it to a faithful linear functional, $\tau(x\cdot).$
We have
$\F(c\convolve y)=\left(\tau(c\cdot)\tensor\tau(y\cdot)\right)\compose\coproduct,$
where $\tau(c\cdot)$ and $\tau(y\cdot)$ are faithful self-adjoint linear functionals.
Since $\tau(c\cdot)\tensor\tau(y\cdot)$ is also a faithful self-adjoint linear functional, we apply Corollary \ref{cor:pullback.of.flf}.
\end{proof}
We notice that $c\convolve\Id=\tau(c)\Id,$ so that if $\tau(c)=1,$ we have a Corollary that we do not use, but is perhaps worth recording:
\comment{because it comes close to giving a ring structure on cuntz. Could I improve to get orthgonality-preserving?}
\begin{corollary}If $c$ is self-adjoint, invertible, and $\tau(c)=1$, then the map $y\mapsto y\convolve c$ is spectrum-preserving on Hermitian operators.\end{corollary}
The next lemma is surely known to experts, but we include the short proof.
 \begin{lemma} Let $f\colon A\arrow B$ be a unital linear bijection of unital C*-algebras that is also a bijection of positive elements and a bijection of centres. 
Then $f(az)=f(a)f(z)$ for all $a\in A$ and any central element $z\in Z(A).$ \label{lem:bimodule.over.centre}\end{lemma}
\begin{proof} Since the map $f$ takes positive elements to positive elements, it therefore takes self-adjoint elements to self-adjoint elements.  The Kadison--Schwartz inequality then gives $f(b^2)\geq f(b)^2$ for all self-adjoint elements $b$ in $A.$ Similar remarks apply to $f\inv$ and therefore the reverse inequality holds as well. It follows that $f(b^2)= f(b)^2.$ Setting $b=a+z,$ and expanding, where the element $a$ is self-adjoint and $z$ is self-adjoint and central, we deduce that $2f(az)=2f(a)f(z).$ We thus have the desired conclusion, at least for self-adjoint elements $a.$ However, since any element can be written as a linear combination of two self-adjoint elements, the general case follows.\end{proof}
In the next lemma, the key point is to show a certain map is positive.
\begin{lemma}If $c$ and $c\inv$ are mapped to invertible self-adjoint operators by the Fourier transform, and if $c$ commutes with co-central elements, then the action induced by $x\mapsto c x c\inv$ on the dual algebra is a direct sum of *-automorphisms and *-anti-automorphisms of matrix blocks.\label{lem:dual.is.Jordan}
\end{lemma}
\begin{proof}The action induced by $x\mapsto c x c\inv$ on the dual will be denoted $\alpha_c.$ Let $a$ and $b$ be the operators of the dual associated with $c$ and $c\inv,$ respectively, by the Fourier transform. 
The map $\alpha_c$ takes an element $y$ to $a\convolve y\convolve b$. By Lemma \ref{lem:sa.inverse.case} we have that the map $\alpha_c$ takes self-adjoint and invertible elements to self-adjoint and invertible elements.  Since the map $x\mapsto cxc\inv$ fixes each element of the co-centre of the algebra, the map $y\mapsto a\convolve y\convolve b$ fixes each element of the centre of the dual algebra. It certainly follows from this that the map is unital, and then the property (from Lemma \ref{lem:sa.inverse.case}) that an invertible self-adjoint operator is mapped to an invertible self-adjoint operator implies that  the map $\alpha_c$ preserves the spectrum of self-adjoint operators. But the fact that self-adjoint elements with positive spectrum are positive as operators implies that $y\mapsto a\convolve y\convolve b$ is a positive map.  
Since the same remarks apply to the inverse map, Lemma \ref{lem:bimodule.over.centre} implies that $\alpha_c(zy)=\alpha_c(z)\alpha_c(y).$    We are now in a position to apply the centre-fixing property of $\alpha_c.$ Taking, then, $z$ to be a central projection, we deduce that   $\alpha_c$ maps a matrix block to itself. 
But a unital linear bijection of matrix algebras having the property of taking positive elements to positive elements is either an inner *-automorphism or an inner *-automorphism composed with the ordinary (non-Hermitian) transpose (this result dates back to \cite{schneider}, see also \cite{semrl}). \end{proof}
The conclusion of the above technical lemma can be rephrased in terms of Jordan *-homomorphisms. 
Recall that a map is a Jordan homomorphism if it maps a square of an arbitrary element to a square. A direct sum of *-automorphisms and *-antiautomorphisms of a C*-algebra is equivalent to a Jordan *-homomorphism. (For matrix algebras, this was shown by Jacobson and Rickart\cite{JR2,JR1}. Kadison\cite{kadison.isometries,kadison2} extended these results for surjective maps onto C*-algebras and von Neumann algebras, and showed also that bijective order isomorphisms are Jordan isomorphisms. St\o rmer showed\cite{stormer1965} that Jordan homomorphisms of C*-algebras into C*-algebras are sums of homomorphisms and anti-homomorphisms.)

The next lemma follows from \cite[Th\'eor\`eme 2.9 ]{CES}.
\begin{lemma} Let $A$ and $B$ be finite-dimensional Hopf C*-algebras with tracial Haar states. Let $\alpha\colon A\arrow B$ be a  *-isomorphism, and let $\hat{\alpha}\colon \Bh\arrow\Ah$ be its action on the dual. We suppose the action $\hat{\alpha}$ on the dual is a Jordan *-isomorphism. Then either $\hat\alpha$ is multiplicative, or $\hat\alpha$ is anti-multiplicative. \label{lem:jordan.auto}\end{lemma}

\begin{proposition} If a \ksymmetric operator $v$ commutes with co-central elements of $A,$ and is invertible, then $x\mapsto vxv\inv$ is a Hopf algebra automorphism or \emph{co-}anti-automorphism.
\label{prop:is.antiauto}\end{proposition}
\begin{proof}Let us first suppose that the Fourier transforms of $v$ and $v\inv$ are invertible operators. We can apply Lemma \ref{lem:dual.is.Jordan} to conclude that the action induced by $x\mapsto vxv\inv$ on the dual is a Jordan *-homomorphism, and then apply Lemma \ref{lem:jordan.auto} to conclude that  this induced action is in fact either a *-automorphism or a *-\textit{co-}anti-automorphism. The uniqueness of the antipode of a Hopf algebra then implies that the bi-algebra (\textit{co-}anti-)automorphism   $x\mapsto vxv\inv$ is in fact a Hopf algebra (\textit{co-}anti-)automorphism.

We now address the case where  the  Fourier transforms of $v$ and $v\inv$ are not both invertible operators.
From the fact that $j$, the support of the co-unit character $\phi$ is a central projection in $A$, we deduce that
$$ (j\epsilon + v)\inv = v\inv -\epsilon j \frac{1}{\phi(v)(\epsilon+\phi(v))}. $$
This equation tells us that if we replace $v$ by $\tilde v = v + \epsilon j,$ then $\tilde v \inv = f(\epsilon) j + v\inv$ where $f$ is a function that vanishes at $\epsilon=0.$ Since, up to a fixed scalar factor, the Fourier transform maps $j$ to  $\Id,$ this allows us to insure that $\F(\tilde{v})$ and $\F(\tilde{v}\inv)$ are both invertible by means of an arbitrarily small perturbation. We note that the necessary condition $\phi(v)\not=0$ is guaranteed by the invertibility of the element $v$ and the fact that $\phi$ is a character. We remark that since $j$ is central, the element $\tilde v$ commutes with co-central elements, and evidently $\tilde v$ remains invertible. Moreover, $j$ is \ksymmetric,  so all the conditions enumerated in the hypothesis remain valid if we replace $v$ by $\tilde v.$ Doing this, we obtain that $x\mapsto \tilde{v}x\tilde{v}\inv$ is a Hopf algebra (\textit{co-}anti-)automorphism, and by the finite-dimensionality of the ambient algebra(s), we obtain $x\mapsto vxv\inv$ as a norm limit of Hopf algebra (\textit{co-}anti-)automorphisms. In particular, this rests upon the fact that in the finite dimensional case, the Fourier transform is continuous. 
\end{proof}

\section{Bi-inner automorphisms: main results}
We now proceed to give two characterizations of the bi-inner Hopf *-automorphisms of a finite-dimensional Hopf C*-algebra. In the first characterization, we view the given Hopf algebra and its dual as concrete sub-algebras, $S,$ and $\Sh,$ of a (finite-dimensional) copy of $\BH.$ There is in this picture a multiplicative unitary $V\in\BHH,$ which can be taken to belong  to $\Sh\tensor S\subset\BHH.$ If this is done, then this unitary is essentially unique, see \cite[Th\'eor\`em 6.2]{BBS}, at least in the finite-dimensional case. We call this the multiplicative unitary associated with the Hopf C*-subalgebras $S$ and $\Sh.$ We then have a Theorem: 
\begin{theorem} Let $S$ and $\Sh$ be concrete Hopf C*-subalgebras associated with a multiplicative unitary $V.$ Let us suppose that the unitary $u$ gives a Hopf algebra automorphism of $S$ whose action on the dual $\Sh$ is inner. There then exists a unitary $\uh \in \Sh$ such that $\uh \tensor u$ commutes with the multiplicative unitary $V.$
\label{th:bi.inner}
 \end{theorem}
\begin{proof} Let $\hat u$ be the unitary coming from the action on the dual.
Since we have $$\beta(u^*xu,\uh^*y\uh)  = \beta(x,  y  ),$$
if we 
replace the multiplicative unitary $V$  by $(\uh^* \tensor u^*)V(\uh \tensor u)$, we notice that $S=\{ (\omega\tensor\Id)V : \omega\in B(\H)^* \}$ and $\Sh=\{ (\Id\tensor\rho)V : \rho\in B(\H)^* \}$ change only up to an inner automorphism, and thus remain invariant as subalgebras of $B(\H).$ It follows from the definition\cite{BBS} of  \fixe and \cofixe vectors that the elements of $S$ map a \cofixe vector to a multiple of itself, and that the elements of $\Sh$ map a \fixe vector to a multiple of itself. Thus, the transformation we have made will not change the \fixe and \cofixe vectors, and moreover, the multiplicity of the multiplicative unitary (which we can take to be 1) will not be changed. 
 By Th\'eor\`em 6.2 in \cite{BBS}, there exists only one multiplicative unitary in $B(\H\tensor\H)$ with these properties. Thus, we conclude that $V=(\uh^*\tensor u^*)V(\uh\tensor u)$, or in other words, that  $\uh \tensor u$ commutes with the multiplicative unitary $V$.  \end{proof}

\begin{remark}\label{rem:bi-inner} {A converse of Theorem \ref{th:bi.inner} also holds: if $\uh \tensor u\in \Sh\tensor S$ commutes with the given multiplicative unitary, then $u$ and $\uh$ determine a dual pair of inner automorphisms. This is because, taking the linear functionals point of view, the pairing $\beta(f,g)$ is given by $\beta(f,g)=(f\tensor g)V,$ for $f\in \Sh^*,$ and $g\in S^*.$ But then $\uh$ and $u$ satisfy $\beta(f(\uh\cdot\uh^*),g(u\cdot u^*))=\beta(f,g).$ } \end{remark}

 \begin{corollary} The bi-inner Hopf *-automorphisms of a finite-dimensional Hopf C*-algebra $A$ are a connected set.\label{cor:is.connected}\end{corollary}
\begin{proof} \comment{Given a unitary $\uh\tensor u$ as in the above theorem, we note that the map $\uh\tensor u \mapsto \uh^r \tensor u^r,$ where $0<r\leq1$ is a real scalar, is compatible with the equivalence relation on tensors given by $\uh\tensor \lambda u=  \lambda\uh\tensor  u.$ The compatibility was not so clear, but the interesting stuff in Blecher and Neal, Open Projs I, makes me think it would not be a problem. However, all we need to do is to start with uh tensor u in the tensor product, then choose arbitarily representatives u and uh, which just means choosing a complex scalar of modulus 1, and then we tensor u^r and uh^r .} Given a unitary $\uh\tensor u$ as in Theorem \ref{th:bi.inner}, notice that $ \uh^r \tensor u^r,$ where $r$ is in $(0,1]$, is a simple tensor commuting with  the given multiplicative unitary, $V.$ By Remark \ref{rem:bi-inner} we thus obtain a bi-inner Hopf *-automorphism, and hence a path of bi-inner automorphisms connecting the given automorphism to the identity.
\end{proof}

\begin{lemma}\label{lem:inner} A *-automorphism of a  finite-dimensional real C*-algebra that is connected to the identity by a path of *-automorphisms is inner. \end{lemma} 
\begin{proof} Since K-theory is   homotopy invariant, the K-theory groups are left unaltered by a *-automorphism that is connected to the identity. It follows from the implicit functoriality of Giordano's classification result that the given *-automorphism is inner equivalent to the identity *-automorphism (see \cite[Corollary 5.3]{giordano}.)   \end{proof}

\begin{theorem}The bi-inner Hopf algebra *-automorphisms of a finite-dimensional Hopf C*-algebra $A$ are precisely the component of the identity of the set of maps of the form
$x\mapsto vxv^*$ where $u$ is a {\ksymmetric}al unitary that commutes with the co-commutative sub-algebra of $A.$\label{th:biU}\end{theorem}
\begin{proof}  We note that if we restrict a bi-inner Hopf algebra *-automorphism to the algebra of {\ksymmetric}al elements, we obtain a $*$-automorphism of the algebra of {\ksymmetric}al elements. The given automorphism is determined by its restriction, because any element of the algebra can be decomposed as $a+ib,$ with $a$ and $b$ belonging to the real algebra of \ksymmetric elements.  By Corollary \ref{cor:is.connected}, the restricted automorphism is connected to the identity and is thus inner by Lemma \ref{lem:inner}.  Therefore, the bi-inner *-automorphisms are at least a connected subset of the set of maps of the given form, $x\mapsto vxv^*,$ where $u$ is a {\ksymmetric}al unitary that necessarily commutes with the co-commutative sub-algebra of $A.$ But on the other hand, Proposition \ref{prop:is.antiauto} shows that maps of the given form are either Hopf automorphisms or Hopf \textit{co-}anti-automorphisms. We note that if the Hopf algebra *-(\textit{co-}\/anti-)automorphism  $x\mapsto vxv^*$  is connected to the identity by a path of maps having the same properties, then in fact it must be a Hopf algebra *-automorphism. These facts  establish the theorem.\end{proof}
It is easy to check that the set given by the above theorem is a Lie group. It would be straighforward and perhaps interesting to obtain further information on the Lie group, using the description given in the above Theorem.
 \begin{remark} The bi-inner Hopf *-automorphisms of a finite-dimensional Hopf C*-algebra are a connected Lie group.\end{remark}

\centerline{{\huge {\rotatebox[]{270}{\textdagger}}\!{\rotatebox[]{90}{\textdagger}}}}

\normalsize
\begin{thebibliography}{99}                      
\bibitem{BBS} S. Baaj, \'E. Blanchard, and G. Skandalis, \textit{Unitaires multiplicatifs en dimension finie et leurs sous-objets}, 
Annales de l'Institut Fourier, \textbf{49}, (1999) 1305--1344
\bibitem{BS} S. Baaj,  and G. Skandalis, \textit{Unitaires multiplicatifs et dualit\'e pour les produits crois\'es de
 $C^*$-alg\`ebres}, 
 Ann. Sci. \'Ecole Norm. Sup., \textbf{4}, (1993), 425--488
\bibitem{CES} J. de Canni\`ere, M. Enock, J-M. Schwartz,
\textit{Sur deux r\'esultats d'analyse harmonique non-commutative: une application de la th\'eorie des alg\`ebres de Kac,}
J. Operator Theory  \textbf{5}  (1981),  171--194
\bibitem{elliott} G.A. Elliott, \textit{On the classification of inductive limits of sequences of semisimple finite-dimensional algebras.} Journal of Algebra (1976) \textbf{38}  29--44.
\bibitem{giordano} T. Giordano, \textit{A classification of approximately finite real C*-algebras}, J. Reine Angew. Math. \textbf{385} (1988),  161--194

\bibitem{JR2} N. Jacobson, C.E. Rickart, \textit{Homomorphisms of Jordan rings of self-adjoint elements,} Trans. Amer. Math. Soc. \textbf{7} (1952) 310--322
 
\bibitem{JR1} N. Jacobson, C.E. Rickart, \textit{Jordan homomorphisms of rings,} Trans. Amer. Math. Soc. \textbf{69} (1950). 479--502. 

\bibitem{kadison2} R.V. Kadison, 
\textit{ A generalized Schwarz inequality and algebraic invariants for operator algebras},
 Ann. of Math. (2) \textbf{56} (1952) 494--503 
\bibitem{kadison.isometries} R.V. Kadison, \textit{Isometries of operator algebras,} Ann. Of Math. (2) \textbf{54} (1951) 325--338. 
 \bibitem{kats} G.I. Kats,
\textit{Certain arithmetic properties of ring groups},
J Funkcional. Anal. i Prilo\v zen., \textbf{6} (1972),  88--90
\bibitem{KP1966} {G.I. Kats and V. Pal'yutkin}, \textit{Finite ring groups}, Trudy Moskovskogo Matemati\v ceskogo Ob\v s\v cestva \textbf{15} (1966), 224--261 
\bibitem{semrl} C.K. Li, L. Rodman,  P. \v{S}emrl,
\textit{ Linear maps on selfadjoint operators preserving invertibility, positive definiteness, numerical range,}
 Canad. Math. Bull. \textbf{46} (2003),  216--228 
\bibitem{paulsen} V. Paulsen, \textit{Completely bounded maps and operator algebras},  Cambridge Studies in Advanced
     Mathematics,  Cambridge, UK, 2002.
\bibitem{pedersen}  G.K. Pedersen, $C^{*}${\it -algebras and their automorphism
groups.\/} London Mathematical Society Monographs, \textbf{14}, Academic Press, Inc.,
London-New York, 1979.
\bibitem{sakai1958} S. Sakai, \textit{On linear functionals of W*-algebras}, Tohoku Math J, \textbf{9}, (1958), 571--574
\bibitem{schneider} H. Schneider, \textit{Positive Operators and an Inertia Theorem}, Numerische Math. \textbf{7} (1965), 11--17
\bibitem{Stinespring} W.F. Stinespring, \textit{Positive functions on {$C^*$}-algebras}, {Proc. Amer. Math. Soc.} \textbf{6} (1955), {211--216}
\bibitem{stormer1965} E. St\o rmer, \textit{On the Jordan structure of C*-algebras.} Trans. Amer. Math. Soc. \textbf{120} (1965) 438--447
\bibitem{VanDaele1996} A. Van Daele, \textit{Discrete Quantum Groups}, J. of Alg., \textbf{180}, (1996) 431--444
\bibitem{VanDaele1998} A. Van Daele, \textit{An Algebraic Framework for Group Duality}, Advances in Math. \textbf{140} (1998), 323--366 
\end{thebibliography}
\end{document}